\documentclass[a4paper,oneside,reqno,12pt]{amsart}

\usepackage[margin=30mm]{geometry}

\usepackage{amsmath, amssymb, amsthm, mathtools}
\usepackage{mathrsfs}
\usepackage{tabularx}
\usepackage{url}

\usepackage[T1]{fontenc}
\usepackage{lmodern}
\usepackage[final]{microtype}

\usepackage[dvipsnames]{xcolor}
\usepackage[pdfencoding=auto,colorlinks]{hyperref}
\hypersetup{
    colorlinks=true,
    linkcolor=BrickRed,
    citecolor=Green,
    filecolor=Mulberry,
    urlcolor=NavyBlue,
    menucolor=BrickRed,
    runcolor=Mulberry,
    pdftitle={Ball codes: A coding characterization of Hausdorff and packing dimensions},
    pdfauthor={Kenshi Miyabe},
    pdfsubject={28A78, 03D32, 68Q30},
    pdfkeywords={Hausdorff dimension, packing dimension,
Kolmogorov complexity, point-to-set principle,
effective fractal dimension},
}

\theoremstyle{plain}
\newtheorem{theorem}{Theorem}[section]
\newtheorem{lemma}[theorem]{Lemma}
\newtheorem{corollary}[theorem]{Corollary}
\newtheorem{proposition}[theorem]{Proposition}

\theoremstyle{definition}
\newtheorem{definition}[theorem]{Definition}

\theoremstyle{remark}
\newtheorem{remark}[theorem]{Remark}

\DeclareMathOperator{\Dim}{Dim}
\DeclareMathOperator{\dom}{dom}
\DeclareMathOperator{\diam}{diam}

\begin{document}

\title[Ball Codes and Dimensions]{Ball codes: A coding characterization of Hausdorff and packing dimensions}

\author{Kenshi Miyabe}

\address[K.~Miyabe]{Meiji University, Japan}
\email{research@kenshi.miyabe.name}

\date{}

\subjclass[2020]{28A78, 03D32, 68Q30}
\keywords{Hausdorff dimension, packing dimension,
Kolmogorov complexity, point-to-set principle,
effective fractal dimension}


\begin{abstract}
We prove a purely classical, coding-theoretic characterization of Hausdorff and
packing dimensions in \(\mathbb R^n\).  A \emph{ball code} assigns names to closed
balls of \(\mathbb R^n\): it is a partial map from a prefix-free set of finite
binary strings to balls.
For every nonempty
\(E\subseteq\mathbb R^n\), the Hausdorff dimension of \(E\) is the minimum, over
all ball codes, of the supremum over \(x\in E\) of the lower asymptotic rate at
which \(x\) can be described by names of balls containing it; the packing
dimension is obtained in the same way from the upper rates.
The Hausdorff characterization
is a Euclidean counterpart of Ryabko's coding theorem for
combinatorial sources.
We obtain the packing characterization from the standard characterization of
packing dimension by upper modified box-counting dimension, without encoding
\(\mathbb R^n\) into a sequence space.
We then derive the point-to-set principle of J.~Lutz and N.~Lutz by
replacing a minimizing ball code with a rational one and storing the resulting
countable codebook in an oracle.
\end{abstract}


\maketitle


\section{Introduction}
\label{sec:introduction}

Hausdorff dimension is defined in terms of covers of a set
\(E\subseteq\mathbb R^n\) whose total \(s\)-cost becomes arbitrarily small as
their diameters tend to zero.  Nothing in the definition refers to the
individual points of \(E\).  Information theory
suggests a different measure of size: fix a scheme for naming small regions of
\(\mathbb R^n\) by finite binary strings, and ask how many bits are needed to name
a region that contains a given point \(x\) and has diameter at most \(2^{-r}\).
The asymptotic number of bits per unit of precision defines a description rate
for \(x\).  We show that both Hausdorff and packing dimensions are given by
such rates on \(\mathbb R^n\).  The point-to-set principle of J.~Lutz and
N.~Lutz~\cite{LutzLutz2018} suggests such a correspondence.  That principle uses
an oracle and an optimal prefix-free machine.  In
Section~\ref{sec:point-to-set-principles}, we derive it from our main result (Theorem~\ref{thm:intro-main}).

We use a \emph{ball code},
which is a partial map \(\varphi\) from finite binary strings to
closed balls of \(\mathbb R^n\) whose domain is prefix-free.  A string
\(p\in\dom\varphi\) is a \emph{\(\varphi\)-name} of the ball
\(\varphi(p)\); a ball may carry several names, or none.  For \(x\in\mathbb R^n\) and
\(r\in\omega\) put
\[
K_\varphi(x;r)
=
\min\bigl\{\,|p| : p\in\dom\varphi,\ x\in\varphi(p),\
\diam\varphi(p)\le 2^{-r}\,\bigr\},
\]
with \(K_\varphi(x;r)=\infty\) if there is no such ball.

Our main results characterize Hausdorff and packing dimensions
as follows.

\begin{theorem}[Theorems~\ref{thm:prefix-free-ball-code-hausdorff}
and~\ref{thm:packing-ball-code-dimension}]
\label{thm:intro-main}
For every nonempty \(E\subseteq\mathbb R^n\),
\begin{equation}
\label{eq:intro-classical}
\dim_H(E)
=
\min_{\varphi}
\sup_{x\in E}
\liminf_{r\to\infty}\frac{K_\varphi(x;r)}{r},
\qquad
\dim_P(E)
=
\min_{\varphi}
\sup_{x\in E}
\limsup_{r\to\infty}\frac{K_\varphi(x;r)}{r},
\end{equation}
where \(\varphi\) ranges over all prefix-free ball codes.  In both cases the
minimum is attained.
\end{theorem}

No machine, no oracle, and no notion of computability occur in
\eqref{eq:intro-classical}:
this is a statement of classical fractal geometry, and
it is proved by classical means.
The two formulas differ only in the use of \(\liminf\) and \(\limsup\), but
this difference leads to substantially different proofs.
A \(\liminf\) only requires short names at infinitely many precisions,
so covers at sparse scales can be assembled with negligible indexing cost.
A \(\limsup\) requires short names at all large precisions,
so the proof must use uniform box-counting control and then the countable
decomposition that turns upper box-counting dimension into packing dimension.

The paper makes two main contributions.
First, it separates the geometric and
counting arguments in the point-to-set principle from the subsequent encoding
of countable data into an oracle.
Second, it gives what is, to our knowledge,
the first classical coding characterization of packing dimension.  The proof is
based on upper modified box-counting dimension (see also Remark~\ref{rem:recovered}).

We remark here that the minimum in
\eqref{eq:intro-classical} cannot be replaced by evaluation at a single universal code: a code minimizing for every singleton would force both right-hand sides to vanish identically.
The code must therefore depend on the set \(E\).

\subsection*{Related work}

The Hausdorff half of Theorem~\ref{thm:intro-main} is a Euclidean, ball-based
counterpart of Ryabko's coding theorem for combinatorial
sources~\cite{Ryabko1986}: for a finite alphabet \(A\) and every
\(I\subseteq A^\omega\),
\[
\min_{\varphi}\ \sup_{x\in I}\ \liminf_{n\to\infty}\frac{v_\varphi(n,x)}{n}
\;=\;
\dim_H(I)\cdot\log_2|A|,
\]
where \(v_\varphi(n,x)\) is the number of code bits needed to recover
\(x\!\upharpoonright\! n\), and the factor \(\log_2|A|\) is the normalization
implicit in counting bits while measuring diameters in cylinders.
Here, a combinatorial source means an arbitrary set of infinite sequences
over a finite alphabet, without an underlying probability distribution.

In Ryabko's formulation, an infinite source sequence is encoded by a single
infinite binary sequence, and finite prefixes of the source are recovered from
finite prefixes of the code.
A ball code, by contrast, is finitary: a finite
string directly names a neighborhood of finite precision, and the names used at
different precisions are not required to be prefixes of one common infinite
name.  The two formulations impose different requirements on codes.
Removing coherence enlarges the class of admissible codes.
Consequently, the existence
of a minimizing code is more demanding in Ryabko's model, whereas the lower
bound for every admissible code is more demanding in ours.
The point-to-set principle requires this lower bound for codes that are not
monotone: the domain of an oracle prefix-free machine has no coherence between
scales.
The finitary formulation also allows computability to be
imposed directly on a partial map from finite strings to rational balls, which
is exactly what a prefix-free machine is.

Beyond this difference in format, two further points separate Theorem~\ref{thm:intro-main} from Ryabko's theorem.  The first
is the passage to \(\mathbb R^n\), where the canonical cell structure of
\(A^\omega\) --- cylinders, nested and pairwise disjoint at each level --- is
absent.  Balls overlap arbitrarily, a point lies in many balls of a given
diameter, and the scales at which a code names balls need not be organized into
levels.  The second difference concerns packing dimension, which is not treated
in Ryabko's theorem.  One might expect it to follow from the gale
characterizations of packing dimension on
Cantor space~\cite{AthreyaHitchcockLutzMayordomo2007} by the usual dualities,
transported to \(\mathbb R^n\) through binary expansions.  We do not proceed that
way.  The proof in Section~\ref{sec:packing-ball-codes} uses the standard equality
of packing dimension with upper modified box-counting dimension, and hence the
countable decompositions supplied by that equality.  Nothing is encoded into a
sequence space, and no gale appears.

We use balls rather than the dyadic cubes into which \(\mathbb R^n\) is usually
partitioned because effective dimension on \(\mathbb R^n\) is defined through
approximation of points by rationals at a given
precision~\cite{LutzMayordomo2008}.
Naming a rational ball is a closely related
coding operation, and this choice makes the passage to oracle machines in
Section~\ref{sec:point-to-set-principles} direct.

Finally, this work belongs to a line of research that gives classical proofs of
results first obtained using computability-theoretic methods.
Orponen~\cite{Orponen2021} gave
combinatorial proofs of two projection theorems
from the preliminary version of the work of N.~Lutz and
Stull~\cite{LutzStull2018,LutzStull2024},
which had originally been proved using effective dimension,
and generalized them to projections onto higher-dimensional subspaces.
Lutz, Qi, and Yu~\cite{LutzQiYu2024} conjecture that
their theorem on Hamel bases likewise
admits a classical proof.
Here we give a classical proof of the point-to-set principle itself, rather than
of one of its applications.

\begin{table}[t]
\centering
\begin{tabularx}{\textwidth}{c|X|X|X}
 & Test & Complexity & Gale \\
\hline
Classical
& Hausdorff and packing measures~\cite{Falconer2014}
& Coding rate on \(A^\omega\)~\cite{Ryabko1986};
  ball codes on \(\mathbb R^n\) (Theorem~\ref{thm:intro-main})
& Gales~\cite{Lutz2003};
  strong gales~\cite{AthreyaHitchcockLutzMayordomo2007} \\
Effective
& Partial randomness for Hausdorff dimension~\cite{Tadaki2002}
& Kolmogorov complexity~\cite{Mayordomo2002,%
AthreyaHitchcockLutzMayordomo2007,LutzMayordomo2008}
& Constructive gales~\cite{Lutz2003b,%
AthreyaHitchcockLutzMayordomo2007}
\end{tabularx}
\caption{Characterizations of Hausdorff dimension and, where available, of
packing dimension.  Ball codes supply the classical complexity entry on
\(\mathbb R^n\), for both dimensions at once; the point-to-set principle is its
effectivization.}
\label{tab:test-complexity-gale}
\end{table}

The three columns of Table~\ref{tab:test-complexity-gale} record the frameworks in
which dimension bounds are commonly expressed --- by tests, by description
length, and by gales --- together with their classical and effective versions
where available.
The classical test entry is the definition of Hausdorff dimension
itself; the classical gale entry is Lutz's gale
characterization~\cite[Theorem~3.6]{Lutz2003}, of which effective Hausdorff
dimension~\cite{Lutz2003b} is the effectivization, the packing analogues of both
being due to Athreya, Hitchcock, Lutz, and
Mayordomo~\cite{AthreyaHitchcockLutzMayordomo2007}.  These entries are recalled in more detail in
Section~\ref{subsec:existing-characterizations}.  The classical
complexity entry on \(\mathbb R^n\) is the one supplied here.

\subsection*{The point-to-set principle}

Granting Theorem~\ref{thm:intro-main}, the point-to-set principle
\begin{equation}
\label{eq:intro-psp}
\dim_H(E)
=
\min_{A\subseteq\omega}
\sup_{x\in E}
\liminf_{r\to\infty}\frac{K^A_r(x)}{r},
\qquad
\dim_P(E)
=
\min_{A\subseteq\omega}
\sup_{x\in E}
\limsup_{r\to\infty}\frac{K^A_r(x)}{r}
\end{equation}
--- where \(K^A_r(x)\) is the length of a shortest program that, given the oracle
\(A\), outputs a rational point within \(2^{-r}\) of \(x\) --- follows from two
steps.
Every prefix-free ball code may be replaced by
one whose values are rational closed balls, at the cost of one step of precision
(Lemma~\ref{lem:rationalize-prefix-free-ball-code}).
The resulting countable codebook can be encoded into an oracle.  An optimal
oracle prefix-free machine can then simulate the code with an additive constant
overhead
(Lemma~\ref{lem:oracle-simulates-rational-ball-code}); conversely, for every
oracle \(A\), the partial map \(U^A\) is a prefix-free ball code.
Minimizing over oracles thus corresponds to minimizing over ball codes, and
\eqref{eq:intro-psp}
follows (Theorem~\ref{thm:point-to-set-coding}).
The formulation of
\eqref{eq:intro-psp} in terms of rational point complexity, which is the one due
to Lutz and Lutz, is then obtained by comparing ball complexity with point
complexity (Proposition~\ref{prop:ball-point-effective-dimensions}) and taking
the resulting asymptotic rates
(Corollary~\ref{cor:ball-point-effective-dimensions}), yielding
Corollary~\ref{cor:rational-point-point-to-set}.

\begin{remark}\label{rem:recovered}
Conversely, granted the point-to-set principle, the rational-ball case of
Theorem~\ref{thm:intro-main} can be recovered by the same two lemmas:
Lemma~\ref{lem:oracle-simulates-rational-ball-code} encodes rational ball codes
into oracles, and Lemma~\ref{lem:rationalize-prefix-free-ball-code} reduces
arbitrary codes to rational ones.  This argument recovers the statement of the
theorem, but not the proof given here.
\end{remark}

The geometric argument in the packing case is closely related to the known
proof of the (extended) point-to-set principle,
which also uses the modified box-counting characterization of packing
dimension~\cite{LutzLutzMayordomo2023}.
The difference is that we first prove
a classical coding theorem and only afterward encode its countable data into an
oracle.

The attainment of the minimum in our classical theorem also yields an oracle
attaining the minimum in the point-to-set principle.  Here, the oracle is
minimizing only in the formula for the dimension of \(E\): it attains the
infimum of the supremum over \(x\in E\).  This should not be confused with the
notion of an \emph{optimal oracle} introduced
by Stull~\cite{Stull2022}, which imposes a substantially stronger requirement and,
when it is met, implies the conclusion of Marstrand's projection theorem for
\(E\).  Whether optimal oracles in that sense admit a classical characterization in
the ball-code model we leave open.

The point-to-set principle has been used in several problems in classical
fractal geometry.  It has been used to extend the fractal intersection formulas from Borel sets to
arbitrary sets~\cite{Lutz2021}, to extend Marstrand's projection theorem to
arbitrary sets whose Hausdorff and packing dimensions agree~\cite{LutzStull2024},
and to construct Hamel bases of every prescribed dimension~\cite{LutzQiYu2024}.
This approach does not replace the genuinely effective tools used in those
proofs---relativized complexity estimates, symmetry of information, and
problem-specific oracle constructions---but rearranges these ingredients to
prove \eqref{eq:intro-psp}.

\subsection*{Organization}

Section~\ref{sec:preliminaries} recalls the classical definitions of Hausdorff,
packing, and box-counting dimension, and the background on algorithmic fractal
dimensions.  Section~\ref{sec:hausdorff-ball-codes} introduces ball codes and
proves the Hausdorff half of Theorem~\ref{thm:intro-main} for prefix-free ball
codes, followed by a remark on plain ball codes.
Section~\ref{sec:packing-ball-codes} proves
the packing half and records the analogous observation for plain ball codes.
Section~\ref{sec:point-to-set-principles} carries out the
effectivization and derives the point-to-set principles, first in coding form and
then in the usual rational-point form.

\section{Preliminaries}
\label{sec:preliminaries}

Throughout the paper, a positive integer \(n\ge1\) is fixed as the dimension
of the underlying Euclidean space $\mathbb{R}^n$.
We use the adjective
\emph{classical} for statements whose formulation involves no notion of
computability, irrespective of when they were proved.

\subsection{Hausdorff dimension and packing dimension}

For the classical theory of Hausdorff and packing dimensions, our main
reference is Falconer~\cite{Falconer2014}.
For a set \(U\subseteq\mathbb R^n\), write
\[
\diam U=\sup\{|x-y|:x,y\in U\}
\]
for its diameter when \(U\ne\emptyset\), and set \(\diam\emptyset=0\).

We begin by recalling the definition of classical Hausdorff dimension.

\begin{definition}[{See \cite[Subsections~3.1--3.2]{Falconer2014}}]
\label{def:hausdorff-dimension}
Let \(E \subseteq \mathbb R^n\) and let \(s \ge 0\).
For any \(\delta > 0\), we define
\[
\mathcal H^s_\delta(E)
=
\inf
\left\{
\sum_{i\in I} (\diam U_i)^s
:
E \subseteq \bigcup_{i\in I} U_i,
\ 0<\diam(U_i) \le \delta
\text{ for every } i\in I
\right\},
\]
where the infimum is taken over all finite or countable, possibly empty, index
sets \(I\) and all families \((U_i)_{i\in I}\) of subsets of \(\mathbb{R}^n\).
The \(s\)-dimensional Hausdorff measure of \(E\) is
\[
\mathcal H^s(E)
=
\lim_{\delta\to0}
\mathcal H^s_\delta(E).
\]
The Hausdorff dimension of \(E\) is
\[
\dim_H(E)
=
\inf\{s\ge0:\mathcal H^s(E)=0\}.
\]
\end{definition}

The requirement \(\diam(U_i)>0\) is harmless, since a set of diameter zero can
be replaced by a ball of arbitrarily small positive diameter.

Notice that \(\mathcal H^s_\delta(\emptyset)=\mathcal{H}^s(\emptyset)=0\)
and \(\dim_H(\emptyset)=0.\)

We also recall the classical definition of packing dimension.

\begin{definition}[{See \cite[Subsection~3.5]{Falconer2014}}]
\label{def:packing-dimension}
Let \(E\subseteq\mathbb R^n\), let \(s\ge0\), and let \(\delta>0\).
A \(\delta\)-packing of \(E\) is a finite or countable, possibly empty, family \((B_i)_i\) of pairwise
disjoint closed balls of positive radius such that each \(B_i\) has center in \(E\) and
radius at most \(\delta\).  Define
\[
\mathcal P^s_\delta(E)
=
\sup
\left\{
\sum_i (\diam B_i)^s:
(B_i)_i \text{ is a }\delta\text{-packing of }E
\right\}.
\]
Then put
\[
\mathcal P^s_0(E)
=
\lim_{\delta\to0}\mathcal P^s_\delta(E).
\]
The \(s\)-dimensional packing measure of \(E\) is
\[
\mathcal P^s(E)
=
\inf
\left\{
\sum_{j=1}^\infty \mathcal P^s_0(E_j):
E\subseteq\bigcup_{j=1}^\infty E_j
\right\}.
\]
The packing dimension of \(E\) is
\[
\dim_P(E)
=
\inf\{s\ge0:\mathcal P^s(E)=0\}.
\]
\end{definition}

Notice that \(\mathcal{P}^s_\delta(\emptyset)=\mathcal{P}^s_0(\emptyset)=0\)
and \(\dim_P(\emptyset)=0\).

We will use the following standard characterization of packing dimension
in terms of upper box-counting dimension.

\begin{definition}[{See \cite[Subsections~2.1 and~2.3]{Falconer2014}}]
\label{def:box-counting-dimensions}
For a nonempty bounded set \(F\subseteq\mathbb R^n\) and \(\delta>0\), let
\(N_\delta(F)\) be the least number of closed balls of radius \(\delta\)
needed to cover \(F\).  The upper box-counting dimension of \(F\) is
\[
\overline{\dim}_{\mathrm B}(F)
=
\limsup_{\delta\to0}
\frac{\log N_\delta(F)}{-\log\delta}.
\]
The base of the logarithm is immaterial in this ratio.

The upper modified box-counting dimension of a set
\(E\subseteq\mathbb R^n\) is defined by
\[
\overline{\dim}_{\mathrm{MB}}(\emptyset)=0,
\]
and, for \(E\ne\emptyset\), by
\[
\overline{\dim}_{\mathrm{MB}}(E)
=
\inf
\left\{
\sup_j \overline{\dim}_{\mathrm B}(E_j):
\begin{array}{l}
E\subseteq\bigcup_{j\in\omega} E_j,\\
E_j\ne\emptyset\text{ and }E_j\text{ is bounded for every }j
\end{array}
\right\}.
\]
\end{definition}

\begin{theorem}[{Box-counting characterization of packing dimension; see \cite[Proposition~3.9]{Falconer2014}}]
\label{thm:packing-modified-box-counting}
For every \(E\subseteq\mathbb R^n\),
\[
\dim_P(E)=\overline{\dim}_{\mathrm{MB}}(E).
\]
\end{theorem}

\subsection{Existing characterizations of algorithmic fractal dimensions}
\label{subsec:existing-characterizations}

For background on algorithmic fractal dimensions we refer to the survey by
Lutz and Mayordomo~\cite{LutzMayordomo2021} and to Downey and
Hirschfeldt~\cite[Chapter~13]{DowneyHirschfeldt2010}.  
 
We now recall the entries of Table~\ref{tab:test-complexity-gale}.  The
classical test entry is the definition of Hausdorff dimension recalled above:
dimension is characterized by the existence of covers whose total \(s\)-cost
is arbitrarily small.  The classical gale entry is Lutz's gale
characterization of Hausdorff dimension on Cantor
space~\cite[Theorem~3.6]{Lutz2003}, and the effective gale entry is effective
Hausdorff dimension, introduced in~\cite{Lutz2003b} by effectivizing it.

The
effective complexity entry is Mayordomo's characterization of effective
Hausdorff dimension in terms of prefix-free Kolmogorov complexity \(K\)~\cite{Mayordomo2002}:
for every \(X\in2^\omega\),
\[
\dim_e(X)
=
\liminf_{r\to\infty}\frac{K(X\upharpoonright r)}{r}.
\]
The effective test entry may be taken to be Tadaki's partial Martin-L\"of
randomness~\cite{Tadaki2002}, which provides the effective-test formulation
matching effective Hausdorff dimension; see also Downey and
Hirschfeldt~\cite[Proposition~13.5.3]{DowneyHirschfeldt2010}.

The
packing-dimension analogues of the gale and effective complexity entries are
due to Athreya, Hitchcock, Lutz, and
Mayordomo~\cite{AthreyaHitchcockLutzMayordomo2007}; in particular effective
packing dimension is characterized by
\[
\Dim_e(X)
=
\limsup_{r\to\infty}\frac{K(X\upharpoonright r)}{r}
\]
for every \(X\in2^\omega\).
 
The effective complexity characterizations were transferred to Euclidean
space by Lutz and Mayordomo~\cite{LutzMayordomo2008}.  For
\(x\in\mathbb R^n\), let
\begin{equation}
\label{eq:euclidean-point-complexity}
K_r(x)
=
\min\{K(q):q\in\mathbb Q^n,\ |x-q|<2^{-r}\}.
\end{equation}
Then define
\begin{equation}
\label{eq:euclidean-effective-dimensions}
\begin{aligned}
\dim_e(x)
&=
\liminf_{r\to\infty}
\frac{K_r(x)}{r},
&
\Dim_e(x)
&=
\limsup_{r\to\infty}
\frac{K_r(x)}{r}.
\end{aligned}
\end{equation}
These are the standard Euclidean effective dimensions defined via rational
point complexity.  In Subsection~\ref{subsec:rational-point-complexity} we
define their relativized versions \(\dim_e^A(x)\) and \(\Dim_e^A(x)\).
Corollary~\ref{cor:ball-point-effective-dimensions} shows that they can
equivalently be expressed using rational closed-ball complexity \(K^A(x;r)\).
For the empty oracle, these definitions recover the quantities
in~\eqref{eq:euclidean-effective-dimensions}.

This leaves the classical complexity entry, which is the subject of the
present paper.  For Cantor space, a coding-theoretic characterization of
Hausdorff dimension for combinatorial sources was obtained by
Ryabko~\cite{Ryabko1986}.
A number of further relations between Hausdorff
dimension and Kolmogorov complexity were surveyed by
Lutz~\cite[Section~6]{Lutz2003b}.
The differences between Ryabko's coding model and
the ball-code model used here
are discussed in the related work section of Section~\ref{sec:introduction}.
The results surveyed in~\cite[Section~6]{Lutz2003b}
concern effective complexity.

Finally, we recall the point-to-set principle, which will be derived in
Section~\ref{sec:point-to-set-principles}.
 
\begin{theorem}[Point-to-set principle of Lutz and Lutz~\cite{LutzLutz2018}]
\label{thm:lutz-lutz-point-to-set}
For every nonempty set \(E\subseteq\mathbb R^n\),
\[
\dim_H(E)
=
\min_{A\subseteq\omega}
\sup_{x\in E}
\liminf_{r\to\infty}\frac{K^A_r(x)}{r},
\qquad
\dim_P(E)
=
\min_{A\subseteq\omega}
\sup_{x\in E}
\limsup_{r\to\infty}\frac{K^A_r(x)}{r},
\]
where \(K^A_r\) denotes the relativization
of~\eqref{eq:euclidean-point-complexity} to the oracle \(A\).
\end{theorem}
 
Theorem~\ref{thm:lutz-lutz-point-to-set} is
recovered in Corollary~\ref{cor:rational-point-point-to-set}.

Lutz, Lutz, and Mayordomo~\cite{LutzLutzMayordomo2023} extended the
point-to-set principle to arbitrary separable metric spaces and to a large
class of gauge families.

Mayordomo~\cite{Mayordomo2025} proved a finite-state point-to-set principle on
\([0,1)\), using information content at finite precision and minimization over
separator enumerators.
Her result concerns only Hausdorff dimension and is derived from the existing effective point-to-set principle.

\section{Ball Codes and Hausdorff Dimension}
\label{sec:hausdorff-ball-codes}

Closed balls provide a common coding language for both dimensions.
For Hausdorff dimension, arbitrary covering sets may be replaced by
balls at the cost of a fixed factor in the \(s\)-sum, while the
box-counting characterization used for packing dimension is already
formulated in terms of ball covers.

\subsection{Prefix-free ball codes and induced complexity}

We now introduce the prefix-free coding-theoretic notions used in the main
Hausdorff characterization.
A ball code should be thought of as a codebook: a
possibly non-computable assignment of binary names to balls.

\begin{definition}
\label{def:prefix-free-ball-code}
A prefix-free ball code in \(\mathbb R^n\) is a map
\[
\varphi:D\to\mathcal B,
\]
where \(D\subseteq2^{<\omega}\) is prefix-free and \(\mathcal B\) is the
collection of all closed balls of positive radius in \(\mathbb R^n\).  Thus
radius-zero degenerate balls are not allowed.
\end{definition}

\begin{definition}
\label{def:prefix-free-ball-code-dimension}
Let \(\varphi\) be a prefix-free ball code.
For \(x\in\mathbb R^n\) and \(r\in\omega\), define
\[
K_\varphi(x;r)
=
\min
\left\{
|p|
:
p\in\dom(\varphi),
x\in \varphi(p),
\ \diam(\varphi(p))\le 2^{-r}
\right\},
\]
with the convention that \(K_\varphi(x;r)=\infty\) if no such \(p\) exists.
Define
\[
\dim_\varphi(x)
=
\liminf_{r\to\infty}
\frac{K_\varphi(x;r)}{r}.
\]
Finally, for nonempty \(E\subseteq\mathbb R^n\), define
\[
\dim_\varphi(E)
=
\sup_{x\in E}\dim_\varphi(x).
\]
\end{definition}

Throughout,
semicolon notation denotes closed-ball complexity at precision \(r\), with
\(K^A(x;r)\) restricted to rational closed balls, while subscript notation such
as \(K^A_r(x)\) denotes rational point complexity at precision \(r\).

\subsection{Hausdorff dimension via prefix-free ball codes}

The main result of this section is the following characterization of Hausdorff
dimension.

\begin{theorem}
\label{thm:prefix-free-ball-code-hausdorff}
For every nonempty \(E\subseteq\mathbb R^n\),
\[
\dim_H(E)=\min_\varphi \dim_\varphi(E),
\]
where the minimum is taken over all prefix-free ball codes \(\varphi\).
\end{theorem}

\medskip
\noindent\textit{Strategy of the proof.}
The proof has the same general form as the Levin--Schnorr theorem in
algorithmic randomness~\cite[Theorem~6.2.3]{DowneyHirschfeldt2010}: the two
inequalities are obtained by translating between \emph{codes} and
\emph{covers}, in the two directions:
\begin{itemize}
\item \emph{From a code to covers} (Lemma~\ref{lem:code-to-cover}).
  Fix \(s\).  A name \(p\) satisfying \(|p|<sr-k\) for a ball of diameter
  \(\le2^{-r}\) contributes at most \(2^{-k}2^{-|p|}\) to the \(s\)-Hausdorff
  sum.  Prefix-freeness turns the sum of these weights into a Kraft sum, so the
  set of points with such short names at infinitely many precisions has
  \(s\)-Hausdorff measure at most \(2^{-k}\).  Letting \(k\to\infty\) gives
  \(\dim_H(E)\le\dim_\varphi(E)\).
\item \emph{From covers to a code} (Theorem~\ref{thm:fixed-s-hausdorff-null-code}
  and Lemma~\ref{lem:cover-to-code}).  At a fixed exponent \(s\),
  \(\mathcal H^s(E)=0\) is equivalent to the existence of a code whose names
  are shorter than \(sr\) by arbitrarily large additive constants.  Applying
  this to a rational sequence \(q_j\downarrow\dim_H(E)\) gives codes
  \(\varphi_j\); prefixing these countably many codebooks by pairwise
  incomparable strings yields a single prefix-free code \(\varphi\) with
  \(\dim_\varphi(x)\le q_j\) for every \(x\in E\) and every \(j\).
\end{itemize}
Lemma~\ref{lem:code-to-cover} gives \(\dim_H(E)\le\inf_\varphi\dim_\varphi(E)\),
and Lemma~\ref{lem:cover-to-code} gives the reverse inequality together with the
attainment of the infimum.

\begin{lemma}[Short names give small Hausdorff measure]
\label{lem:short-names-small-measure}
Let \(\varphi\) be a prefix-free ball code, let \(s\ge0\), and let
\(k\in\omega\).  Put
\[
F_k^s(\varphi)
=
\bigl\{x\in\mathbb R^n:
K_\varphi(x;r)<sr-k\text{ for infinitely many }r\bigr\}.
\]
Then
\[
\mathcal H^s(F_k^s(\varphi))\le 2^{-k}.
\]
\end{lemma}

\begin{proof}
For \(m\in\omega\), let
\[
D_{k,m}
=
\bigl\{\,p\in\dom(\varphi):
\text{ for some }r\ge m,
\diam(\varphi(p))\le2^{-r}\text{ and }|p|<sr-k\,\bigr\}.
\]
If \(p\in D_{k,m}\) and \(r\) witnesses this, then
\[
(\diam\varphi(p))^s
\le 2^{-sr}
< 2^{-k}2^{-|p|}.
\]
Since \(\dom(\varphi)\) is prefix-free,
\[
\sum_{p\in D_{k,m}}(\diam\varphi(p))^s
\le 2^{-k}\sum_{p\in D_{k,m}}2^{-|p|}
\le 2^{-k}.
\]
All \(\varphi(p)\) for \(p\in D_{k,m}\) have diameter at most \(2^{-m}\).

The family \(\{\varphi(p):p\in D_{k,m}\}\) covers \(F_k^s(\varphi)\).  Indeed,
if \(x\in F_k^s(\varphi)\), choose a witnessing precision \(r\ge m\) and a
shortest name \(p\) for \(x\) at precision \(r\).  Then \(p\in D_{k,m}\) and
\(x\in\varphi(p)\).  Hence
\[
\mathcal H^s_{2^{-m}}(F_k^s(\varphi))
\le
\sum_{p\in D_{k,m}}(\diam\varphi(p))^s
\le 2^{-k}.
\]
Letting \(m\to\infty\) gives the result.
\end{proof}

\begin{lemma}[From a code to covers]
\label{lem:code-to-cover}
For every nonempty \(E\subseteq\mathbb R^n\) and every prefix-free ball code
\(\varphi\),
\[
\dim_H(E)\le \dim_\varphi(E).
\]
\end{lemma}

\begin{proof}
Put \(\alpha=\dim_\varphi(E)\).  If \(\alpha=\infty\), there is nothing to
prove.  Fix \(s>\alpha\).  For \(k\in\omega\), let \(F_k=F_k^s(\varphi)\).
By Lemma~\ref{lem:short-names-small-measure},
\[
\mathcal H^s(F_k)\le2^{-k}.
\]
If \(x\in E\), then \(\dim_\varphi(x)\le\alpha<s\).  For each fixed \(k\), choose
\(\varepsilon>0\) with \(\dim_\varphi(x)<s-\varepsilon\); along infinitely many
sufficiently large \(r\), we have \(K_\varphi(x;r)<(s-\varepsilon)r<sr-k\).
Thus \(x\in F_k\), so \(E\subseteq F_k\) for every \(k\).  Therefore
\[
\mathcal H^s(E)\le\mathcal H^s(F_k)\le2^{-k}
\]
for every \(k\).  Letting \(k\to\infty\), we get \(\mathcal H^s(E)=0\).  Since
this holds for every \(s>\alpha\), \(\dim_H(E)\le\alpha\).
\end{proof}

For the proof from covers to a code,
we use the following converse to Kraft's inequality from information theory;
see~\cite[Theorem~5.2.2]{CoverThomas2006}.
The effective version is called the KC theorem in algorithmic randomness;
see~\cite[Theorem~3.6.1]{DowneyHirschfeldt2010}.
We say that two strings are incomparable if neither is a prefix of the other;
a set of pairwise incomparable strings is therefore prefix-free.

\begin{lemma}[Extended Kraft's inequality]
\label{lem:kraft}
Let \((\ell_i)_{i\in\omega}\) be a sequence of natural numbers such that
\(\sum_{i}2^{-\ell_i}\le 1\).
Then there is a pairwise distinct and pairwise incomparable sequence
\((p_i)_{i\in\omega}\) in \(2^{<\omega}\) with \(|p_i|=\ell_i\) for every
\(i\).
\end{lemma}

\begin{lemma}[Small Hausdorff measure gives small requests]
\label{lem:hausdorff-requests}
Let \(E\subseteq\mathbb R^n\) be nonempty, let \(s>0\), let \(k\ge0\) and
\(m\ge1\) be integers with \(sm\ge k\), and let \(\eta>0\).  If
\[
\mathcal H^s_{2^{-(m+2)}}(E)<2^{-s-k}\eta,
\]
then there are a finite or countable index set \(I\), closed balls
\((B_i)_{i\in I}\), and natural numbers \((\ell_i)_{i\in I}\) such that
\[
E\subseteq\bigcup_{i\in I} B_i,
\qquad
0<\diam(B_i)<2^{-m},
\qquad
\sum_{i\in I}2^{-\ell_i}<\eta,
\]
and
\[
\ell_i<s\bigl(-\log_2\diam(B_i)\bigr)+1-k
\]
for every \(i\in I\).  Thus small \(s\)-Hausdorff content at scale
\(2^{-(m+2)}\) gives a cover whose balls can be requested with small total
Kraft weight and with a prescribed saving \(k\) in the requested lengths.
\end{lemma}

\begin{proof}
Choose a finite or countable cover \((U_i)_{i\in I}\) of \(E\) with
\(0<\diam(U_i)\le2^{-(m+2)}\) and
\[
\sum_{i\in I} (\diam U_i)^s<2^{-s-k}\eta.
\]
For each \(i\), choose \(a_i\in U_i\) and put
\(B_i=\overline B(a_i,\diam U_i)\).  Then \(B_i\supseteq U_i\),
\(0<\diam(B_i)<2^{-m}\), and
\[
\sum_{i\in I}(\diam B_i)^s
=2^s\sum_{i\in I}(\diam U_i)^s
<2^{-k}\eta.
\]
Set
\[
\ell_i=\bigl\lceil -s\log_2\diam(B_i)\bigr\rceil-k .
\]
Since \(\diam(B_i)<2^{-m}\) and \(sm\ge k\),
\[
-s\log_2\diam(B_i)>sm\ge k,
\]
so \(\ell_i\in\omega\).  The inequality \(\lceil t\rceil<t+1\) gives the
required upper bound on \(\ell_i\).  Moreover,
\[
\sum_{i\in I}2^{-\ell_i}
\le
2^k\sum_{i\in I}(\diam B_i)^s
<\eta.
\]
\end{proof}

\begin{theorem}[Hausdorff nullity at a fixed exponent]
\label{thm:fixed-s-hausdorff-null-code}
Let \(s>0\) and let \(E\subseteq\mathbb R^n\) be nonempty.  Then
\[
\mathcal H^s(E)=0
\]
if and only if there is a prefix-free ball code \(\varphi\) such that, for every
\(x\in E\),
\[
\liminf_{r\to\infty}\bigl(K_\varphi(x;r)-sr\bigr)=-\infty.
\]
\end{theorem}

\begin{proof}
First suppose that such a code \(\varphi\) exists.  For each \(k\in\omega\),
the assumed lower limit gives \(E\subseteq F_k^s(\varphi)\).  Hence
Lemma~\ref{lem:short-names-small-measure} gives
\[
\mathcal H^s(E)\le\mathcal H^s(F_k^s(\varphi))\le2^{-k}
\]
for every \(k\), and therefore \(\mathcal H^s(E)=0\).

Conversely, suppose that \(\mathcal H^s(E)=0\).

\smallskip
\noindent\textbf{Step 1 (requests for each pair \((k,m)\)).}
For each pair of integers
\((k,m)\) with \(k\ge0\), \(m\ge1\), and \(sm\ge k\), apply
Lemma~\ref{lem:hausdorff-requests} with
\[
\eta_{k,m}=2^{-k-m-2}.
\]
This is possible because \(\mathcal H^s_{2^{-(m+2)}}(E)=0\).  We obtain a finite
or countable index set \(I_{k,m}\), balls \((B_{k,m,i})_{i\in I_{k,m}}\), and
requested lengths \((\ell_{k,m,i})_{i\in I_{k,m}}\) such that
\[
E\subseteq\bigcup_{i\in I_{k,m}}B_{k,m,i},
\qquad
0<\diam(B_{k,m,i})<2^{-m},
\qquad
\sum_{i\in I_{k,m}}2^{-\ell_{k,m,i}}\le2^{-k-m-2},
\]
and
\[
\ell_{k,m,i}<s\bigl(-\log_2\diam(B_{k,m,i})\bigr)+1-k
\]
for every \(i\in I_{k,m}\).

\smallskip
\noindent\textbf{Step 2 (one code via Kraft).}
The set of all triples \((k,m,i)\) with \(k\ge0\), \(m\ge1\), \(sm\ge k\), and
\(i\in I_{k,m}\) is countable, and the total Kraft weight of all requests is at
most
\[
\sum_{k=0}^\infty\sum_{m=1}^\infty
\sum_{i\in I_{k,m}}2^{-\ell_{k,m,i}}
\le
\sum_{k=0}^\infty\sum_{m=1}^\infty 2^{-k-m-2}\le1.
\]
After enumerating these triples, Lemma~\ref{lem:kraft} gives pairwise
incomparable strings \((p_{k,m,i})\) with
\(|p_{k,m,i}|=\ell_{k,m,i}\).  Define a prefix-free ball code \(\varphi\) by
\[
\varphi(p_{k,m,i})=B_{k,m,i}.
\]

\smallskip
\noindent\textbf{Step 3 (rate estimate at a fixed point).}
Fix \(x\in E\) and \(L>0\).  Choose an integer \(k>s+1+L\).  For every
sufficiently large \(m\) with \(sm\ge k\), choose \(i(m)\in I_{k,m}\) such that
\(x\in B_{k,m,i(m)}\).  Put
\[
\rho_m=\diam(B_{k,m,i(m)}),
\qquad
r_m=\bigl\lfloor -\log_2\rho_m\bigr\rfloor .
\]
Then \(r_m\ge m\), so \(r_m\to\infty\), and \(\rho_m\le2^{-r_m}\).  Hence
\[
K_\varphi(x;r_m)
\le \ell_{k,m,i(m)}
< s(-\log_2\rho_m)+1-k
< sr_m+s+1-k
< sr_m-L.
\]
Since \(L>0\) was arbitrary,
\[
\liminf_{r\to\infty}\bigl(K_\varphi(x;r)-sr\bigr)=-\infty.
\]
This holds for every \(x\in E\).
\end{proof}

\begin{lemma}[From covers to a code]
\label{lem:cover-to-code}
For every nonempty \(E\subseteq\mathbb R^n\) there is a prefix-free ball code \(\varphi\)
such that \(\dim_\varphi(x)\le \dim_H(E)\) for every \(x\in E\); in particular
\(\dim_\varphi(E)\le\dim_H(E)\).
\end{lemma}

\begin{proof}
Put \(d=\dim_H(E)\), and choose a decreasing sequence
\((q_j)_{j\in\omega}\) of rational numbers such that \(q_j>d\) for every \(j\)
and \(q_j\to d\).  For each \(j\), we have \(\mathcal H^{q_j}(E)=0\).
By Theorem~\ref{thm:fixed-s-hausdorff-null-code}, choose a prefix-free ball code
\(\varphi_j\) such that, for every \(x\in E\),
\[
\liminf_{r\to\infty}\bigl(K_{\varphi_j}(x;r)-q_jr\bigr)=-\infty.
\]
Put \(\tau_j=1^j0\).  The strings \((\tau_j)_{j\in\omega}\) are pairwise
incomparable.
Define a prefix-free ball code \(\varphi\) by
\[
\varphi(\tau_jp)=\varphi_j(p)
\]
for \(j\in\omega\) and \(p\in\dom(\varphi_j)\).
This is prefix-free because the strings \(\tau_j\) are pairwise incomparable
and each \(\dom(\varphi_j)\) is prefix-free.

Fix \(x\in E\) and \(j\in\omega\).  For all \(r\),
\[
K_\varphi(x;r)\le K_{\varphi_j}(x;r)+|\tau_j|.
\]
Therefore
\[
\liminf_{r\to\infty}\bigl(K_\varphi(x;r)-q_jr\bigr)=-\infty,
\]
and in particular \(\dim_\varphi(x)\le q_j\).  Since this holds for every
\(j\) and \(q_j\to d\), we have \(\dim_\varphi(x)\le d\).  Hence
\(\dim_\varphi(E)\le d=\dim_H(E)\).
\end{proof}

\begin{proof}[Proof of Theorem~\ref{thm:prefix-free-ball-code-hausdorff}]
Lemma~\ref{lem:code-to-cover} gives
\(\dim_H(E)\le\dim_\varphi(E)\) for every prefix-free ball code \(\varphi\),
while Lemma~\ref{lem:cover-to-code} gives a code \(\varphi^*\) with
\(\dim_{\varphi^*}(E)\le\dim_H(E)\).  Hence equality holds and the minimum is
attained by \(\varphi^*\).
\end{proof}

\begin{remark}
The prefix-free requirement can be omitted from the Hausdorff characterization.
Indeed, let a plain ball code be defined in the same way,
but with an arbitrary domain \(D\subseteq 2^{<\omega}\), and let
\(C_\varphi(x;r)\) denote the corresponding minimum description length.
Every plain code \(\varphi\) can be transformed into a prefix-free code
\(\widehat\varphi\) by replacing each name \(p\) with a prefix-free
encoding \(\widehat p\) satisfying
\[
  |\widehat p|=|p|+O(\log(|p|+2)).
\]
Consequently,
\[
  K_{\widehat\varphi}(x;r)
  \le C_\varphi(x;r)
     +O(\log(C_\varphi(x;r)+2)).
\]
This sublinear overhead does not affect the lower asymptotic rate.
Since every prefix-free ball code is also a plain ball code,
the minimum in the Hausdorff characterization is unchanged if plain ball codes
are used instead.
\end{remark}

\section{Packing Dimension via Ball Codes}
\label{sec:packing-ball-codes}

We now give the corresponding characterization of packing dimension.  At the
level of the induced pointwise rates, the lower limit in the Hausdorff case is
replaced by an upper limit.  The proof correspondingly uses the modified
box-counting characterization of packing dimension rather than Hausdorff
covers.

\begin{definition}
\label{def:packing-ball-code-dimension}
Let \(\varphi\) be a prefix-free ball code.  For \(x\in\mathbb R^n\), define
\[
\Dim_\varphi(x)
=
\limsup_{r\to\infty}
\frac{K_\varphi(x;r)}{r}.
\]
For nonempty \(E\subseteq\mathbb R^n\), put
\[
\Dim_\varphi(E)
=
\sup_{x\in E}\Dim_\varphi(x).
\]
\end{definition}

\begin{theorem}
\label{thm:packing-ball-code-dimension}
For every nonempty \(E\subseteq\mathbb R^n\),
\[
\dim_P(E)=\min_\varphi \Dim_\varphi(E)
\]
where the minimum is taken over all prefix-free ball codes \(\varphi\).
\end{theorem}

\medskip
\noindent\textit{Strategy of the proof.}
The proof runs parallel to that of
Theorem~\ref{thm:prefix-free-ball-code-hausdorff}, with Hausdorff measure
replaced by the \emph{modified box-counting} characterization of packing
dimension (Theorem~\ref{thm:packing-modified-box-counting}):
\[
\dim_P(E)=\inf\Bigl\{\ \sup_j \overline{\dim}_{\mathrm B}(E_j)\ :\
E\subseteq\bigcup_{j}E_j,\ \text{each }E_j\text{ nonempty and bounded}\Bigr\}.
\]
\begin{itemize}
\item \emph{From a code to a decomposition} (Lemma~\ref{lem:packing-code-to-box}).
  If \(\Dim_\varphi(E)<t\), then every \(x\in E\) satisfies \(K_\varphi(x;r)\le tr\)
  for \emph{all} large \(r\) (this is where \(\limsup\) replaces \(\liminf\)).
  Splitting \(E\) according to the threshold from which this holds, and
  intersecting with balls to get boundedness, produces countably many pieces
  \(E_{M,k}\); on each of them the strings of length \(\le tr\) give a cover by at
  most \(2^{tr+1}\) sets of diameter \(\le 2^{-r}\), so
  \(\overline{\dim}_{\mathrm B}(E_{M,k})\le t\).
\item \emph{From upper box-counting dimension to requests}
  (Lemma~\ref{lem:upper-box-to-requests}).  If
  \(\overline{\dim}_{\mathrm B}(F)<q\), then at every sufficiently large scale
  \(2^{-r}\) the set \(F\) has a cover whose balls may be requested with lengths
  \(qr+2\log_2(r+1)+O(1)\), and the total Kraft weight of all these requests can
  be made arbitrarily small.  The logarithmic overhead makes the sum over all scales
  converge without introducing a second exponent.
\item \emph{From packing decompositions to a code}
  (Lemma~\ref{lem:packing-box-to-code}).
  For a rational sequence \(q_j\downarrow\dim_P(E)\),
  decompose \(E\) into bounded pieces of upper box-counting dimension
  \(<q_j\), apply the preceding upper-box lemma to each piece with summable Kraft budgets,
  and then use Lemma~\ref{lem:kraft} once to realize all requests in a
  single prefix-free code.
  This code satisfies \(\Dim_\varphi(x)\le q_j\) for
  every \(x\in E\) and every \(j\).
\end{itemize}

We use the covering number \(N_\delta(F)\) from
Definition~\ref{def:box-counting-dimensions}.  Equivalently, one may use covers
by sets of diameter at most \(\delta\), since this only changes the scale by a
fixed multiplicative factor and hence does not change the upper box-counting dimension.
We use the extended Kraft inequality (Lemma~\ref{lem:kraft}) again.

\begin{lemma}[From a code to a decomposition]
\label{lem:packing-code-to-box}
For every nonempty \(E\subseteq\mathbb R^n\) and every prefix-free ball code
\(\varphi\),
\[
\dim_P(E)\le \Dim_\varphi(E).
\]
\end{lemma}

\begin{proof}
Put \(\alpha=\Dim_\varphi(E)\); we may assume \(\alpha<\infty\).  Fix \(t>\alpha\);
it suffices to prove \(\dim_P(E)\le t\).

\smallskip
\noindent\textbf{Step 1 (decomposition of \(E\)).}
For \(k\in\omega\) put
\[
\begin{aligned}
E_k
&=
\bigl\{\,x\in E:
K_\varphi(x;r)\le tr
\text{ for every }r\ge k\,\bigr\},\\
E_{M,k}
&=
E_k\cap\overline B(0,M)
\qquad(M\in\omega).
\end{aligned}
\]
For every \(x\in E\) we have
\(\limsup_{r\to\infty}K_\varphi(x;r)/r=\Dim_\varphi(x)\le\alpha<t\), so
\(K_\varphi(x;r)\le tr\) holds for all sufficiently large \(r\), i.e.\ \(x\in E_k\)
for some \(k\).  Hence
\begin{equation}
\label{eq:packing-code-decomposition}
E=\bigcup_{M,k\in\omega}E_{M,k}.
\end{equation}
The nonempty members of this family form a countable decomposition of \(E\) into
nonempty bounded sets, as required in
Definition~\ref{def:box-counting-dimensions}.

\smallskip
\noindent\textbf{Step 2 (box-counting on each piece).}
Fix \(M,k\) with \(E_{M,k}\ne\emptyset\), and let \(r\ge k\).  Put
\[
D_r=\bigl\{\,p\in\dom(\varphi)\;:\;|p|\le tr,\ \diam(\varphi(p))\le 2^{-r}\,\bigr\}.
\]
If \(x\in E_{M,k}\) then \(K_\varphi(x;r)\le tr\), so some \(p\in D_r\) satisfies
\(x\in\varphi(p)\).  Thus \(\{\varphi(p):p\in D_r\}\) covers \(E_{M,k}\) by balls of
diameter \(\le 2^{-r}\).  Enlarging each such ball to radius \(2^{-r}\), if
necessary, gives a cover by closed balls of the radius used in the definition of
\(N_{2^{-r}}\).  Since the number of binary strings of length \(\le tr\) is at
most \(2^{tr+1}\),
\[
N_{2^{-r}}(E_{M,k})\le \# D_r\le 2^{tr+1}
\qquad(r\ge k).
\]
By the standard restriction to dyadic scales in the definition of upper
box-counting dimension~\cite[Subsection~2.1]{Falconer2014},
\[
\overline{\dim}_{\mathrm B}(E_{M,k})
=\limsup_{r\to\infty}\frac{\log_2 N_{2^{-r}}(E_{M,k})}{r}
\le\limsup_{r\to\infty}\frac{tr+1}{r}=t .
\]

\smallskip
\noindent\textbf{Step 3 (conclusion).}
Since \(E\ne\emptyset\), the set of pairs \((M,k)\) with
\(E_{M,k}\ne\emptyset\) is nonempty.  Enumerating these nonempty pieces and using
Definition~\ref{def:box-counting-dimensions} together with
Theorem~\ref{thm:packing-modified-box-counting},
\[
\dim_P(E)
=\overline{\dim}_{\mathrm{MB}}(E)
\le\sup_{M,k:\, E_{M,k}\ne\emptyset}
\overline{\dim}_{\mathrm B}(E_{M,k})
\le t .
\]
As \(t>\alpha\) was arbitrary, \(\dim_P(E)\le\alpha=\Dim_\varphi(E)\).
\end{proof}

\begin{lemma}[Upper box-counting dimension gives all-scale requests]
\label{lem:upper-box-to-requests}
Let \(F\subseteq\mathbb R^n\) be nonempty and bounded, and let \(q>0\) satisfy
\[
\overline{\dim}_{\mathrm B}(F)<q.
\]
For every \(\eta>0\), there are \(r_0\in\omega\), closed balls
\(B_{r,i}\ (r\ge r_0, i<m_r)\), and natural numbers
\(\ell_{r,i}\ (r\ge r_0, i<m_r)\) such that, for every \(r\ge r_0\),
\[
F\subseteq\bigcup_{i<m_r}B_{r,i},
\qquad
\diam(B_{r,i})\le2^{-r},
\]
and, in total,
\[
\sum_{r\ge r_0}\sum_{i<m_r}2^{-\ell_{r,i}}<\eta,
\]
and
\[
\ell_{r,i}\le qr+2\log_2(r+1)+O(1)
\]
where the constant in \(O(1)\) may depend on \(F,q,\eta\), but not on
\(r\) or \(i\).  Thus a bounded set of upper box-counting dimension \(<q\) admits
descriptions at \emph{every sufficiently large scale} with rate \(q\), up to a
logarithmic overhead, and with arbitrarily small total Kraft weight.
\end{lemma}

\begin{proof}
By the definition of upper box-counting dimension
(Definition~\ref{def:box-counting-dimensions}), there is \(r_0\in\omega\)
such that, for every \(r\ge r_0\), the set \(F\) is covered by
\(m_r\) closed balls of radius \(2^{-(r+1)}\), where
\[
m_r\le2^{q(r+1)}.
\]
Denote these balls by \(B_{r,i}\ (i<m_r)\).  Each has diameter at most
\(2^{-r}\).

Choose \(c\in\omega\) so large that
\[
2^{q-c}\sum_{r\ge r_0}(r+1)^{-2}<\eta.
\]
Set
\[
\ell_{r,i}=\lceil qr\rceil+2\lceil\log_2(r+1)\rceil+c.
\]
Then
\[
\begin{aligned}
\sum_{r\ge r_0}\sum_{i<m_r}2^{-\ell_{r,i}}
&\le
\sum_{r\ge r_0}
2^{q(r+1)}2^{-qr}(r+1)^{-2}2^{-c} \\
&\le
2^{q-c}\sum_{r\ge r_0}(r+1)^{-2}
<\eta,
\end{aligned}
\]
and \(\ell_{r,i}\le qr+2\log_2(r+1)+O(1)\), as required.
\end{proof}

\begin{lemma}[From packing decompositions to a code]
\label{lem:packing-box-to-code}
For every nonempty \(E\subseteq\mathbb R^n\) there is a prefix-free ball code \(\varphi\)
such that \(\Dim_\varphi(x)\le\dim_P(E)\) for every \(x\in E\); in particular
\(\Dim_\varphi(E)\le\dim_P(E)\).
\end{lemma}

\begin{proof}
Put \(d=\dim_P(E)\), and choose a decreasing sequence
\((q_j)_{j\in\omega}\) of rational numbers such that \(q_j>d\) for every \(j\)
and \(q_j\to d\).  For each \(j\),
Theorem~\ref{thm:packing-modified-box-counting} gives
\(\overline{\dim}_{\mathrm{MB}}(E)<q_j\).  Since \(E\ne\emptyset\), the definition
of upper modified box-counting dimension yields nonempty bounded sets
\((E_{j,k})_{k\in\omega}\) such that
\[
E\subseteq\bigcup_kE_{j,k},
\qquad
\sup_k\overline{\dim}_{\mathrm B}(E_{j,k})<q_j.
\]
Apply Lemma~\ref{lem:upper-box-to-requests} to each \(E_{j,k}\), with
\[
\eta=2^{-j-k-2}.
\]
This gives, for each pair \((j,k)\), scales \(r_{j,k}\), balls \(B_{j,k,r,i}\),
and requested lengths \(\ell_{j,k,r,i}\) such that the balls cover \(E_{j,k}\) at
every \(r\ge r_{j,k}\), have diameter at most \(2^{-r}\), satisfy
\[
\sum_{r\ge r_{j,k}}\sum_i2^{-\ell_{j,k,r,i}}
<2^{-j-k-2},
\]
and obey the following bound for some constant \(c_{j,k}\):
\[
\ell_{j,k,r,i}\le q_jr+2\log_2(r+1)+c_{j,k}.
\]
Consequently,
\[
\sum_j\sum_k\sum_{r\ge r_{j,k}}\sum_i
2^{-\ell_{j,k,r,i}}
<
\sum_j\sum_k2^{-j-k-2}
=1.
\]
By Lemma~\ref{lem:kraft}, there are pairwise incomparable strings
\((p_{j,k,r,i})\) with
\[
|p_{j,k,r,i}|=\ell_{j,k,r,i}.
\]
Define a prefix-free ball code \(\varphi\) by
\[
\varphi(p_{j,k,r,i})=B_{j,k,r,i}.
\]

Fix \(x\in E\) and \(j\in\omega\).  Choose \(k\) with \(x\in E_{j,k}\).  For every
\(r\ge r_{j,k}\), some ball \(B_{j,k,r,i(r)}\) contains \(x\), has diameter at
most \(2^{-r}\), and is coded by \(\varphi\).  Hence, for every \(r\ge r_{j,k}\),
\[
K_\varphi(x;r)
\le
\ell_{j,k,r,i(r)}
\le q_jr+2\log_2(r+1)+c_{j,k}.
\]
The logarithmic and constant terms vanish after division by \(r\), so
\(\Dim_\varphi(x)\le q_j\).  Since this holds for every \(j\) and \(q_j\to d\),
we have \(\Dim_\varphi(x)\le d\).  Hence
\(\Dim_\varphi(E)\le d=\dim_P(E)\).
\end{proof}

\begin{proof}[Proof of Theorem~\ref{thm:packing-ball-code-dimension}]
Lemma~\ref{lem:packing-code-to-box} gives
\(\dim_P(E)\le\Dim_\varphi(E)\) for every prefix-free ball code \(\varphi\), while
Lemma~\ref{lem:packing-box-to-code} gives a prefix-free ball code \(\varphi^*\)
with \(\Dim_{\varphi^*}(E)\le\dim_P(E)\).  Hence equality holds, and
\(\varphi^*\) attains the minimum.
\end{proof}

\begin{remark}
As in the Hausdorff characterization, the prefix-free requirement can be
omitted from the packing characterization without changing the minimum.
\end{remark}

\section{The Point-to-Set Principles}
\label{sec:point-to-set-principles}

\subsection{Rational closed ball complexity}

In order to effectivize closed-ball codes,
we restrict attention to rational closed balls,
which form an effectively enumerable countable class.

\begin{definition}
\label{def:rational-closed-ball}
A rational closed ball is a closed ball
\[
\overline B(q,\rho)\subseteq \mathbb R^n
\]
with \(q\in\mathbb Q^n\) and \(\rho\in\mathbb Q_{>0}\).
\end{definition}

For the asymptotic minimization problems considered below, restricting the
range to rational closed balls is harmless: every prefix-free ball code can
be replaced by a rational one with only a one-step loss in precision.

\begin{lemma}
\label{lem:rationalize-prefix-free-ball-code}
For every prefix-free ball code \(\varphi\), there is a prefix-free ball code
\(\varphi'\) whose values are rational closed balls such that
\[
K_{\varphi'}(x;r)
\le
K_\varphi(x;r+1)
\]
for all \(x\in\mathbb R^n\) and \(r\in\omega\).  In particular, in the
prefix-free ball-code characterizations of Hausdorff and packing dimensions in
Theorems~\ref{thm:prefix-free-ball-code-hausdorff}
and~\ref{thm:packing-ball-code-dimension}, it is enough to consider ball codes
whose ranges consist of rational closed balls.
\end{lemma}

\begin{proof}
For each \(p\in\dom(\varphi)\), choose a rational closed ball \(B_p\) such that
\(\varphi(p)\subseteq B_p\) and
\(\diam(B_p)<2\diam(\varphi(p))\), and define \(\varphi'(p)=B_p\) on the same
domain.  If \(p\) witnesses \(K_\varphi(x;r+1)\), then the same string witnesses
\(K_{\varphi'}(x;r)\le K_\varphi(x;r+1)\).
Taking lower or upper asymptotic rates removes the one-step shift, since
\((r+1)/r\to1\).
\end{proof}

\subsection{Point-to-set principles via rational closed ball complexity}

We now pass from arbitrary prefix-free ball codes to effective ones.
An oracle can encode an arbitrary prefix-free rational ball code.

Fix an optimal oracle prefix-free Turing machine \(U\) whose outputs are
rational closed balls.
Such a machine exists by the usual universal-machine
construction for prefix-free machines; see Downey and
Hirschfeldt~\cite[Proposition~3.5.1]{DowneyHirschfeldt2010}.
For an oracle prefix-free Turing machine \(M\), write
\[
K_M^A(b)
=
\min\{|p|:M^A(p)=b\}
\]
for the complexity of a rational closed ball \(b\), with the convention
that the minimum is \(\infty\) if no such \(p\) exists.  Here optimal means that
for every oracle prefix-free Turing machine \(M\) whose outputs are rational
closed balls, there is a constant \(c_M\) such that
\[
K_U^A(b)\le K_M^A(b)+c_M
\]
for every oracle \(A\subseteq\omega\) and every rational closed ball \(b\).
For an oracle \(A\), the machine \(U^A\) is an oracle-computable prefix-free ball
code whose values are rational closed balls.

\begin{definition}
\label{def:effective-ball-complexity}
For an oracle \(A\subseteq\omega\), \(x\in\mathbb R^n\), and \(r\in\omega\), define
\[
K^A(x;r)
=
\min\left\{
|p|:
p\in\dom(U^A),
x\in U^A(p),\
\diam(U^A(p))\le 2^{-r}
\right\}.
\]
\end{definition}

The set in this minimum is nonempty for all \(A\), \(x\), and \(r\), by
comparison with a prefix-free machine whose range contains every rational
closed ball.

Subsection~\ref{subsec:rational-point-complexity} compares this rational
closed-ball complexity with the standard rational point complexity
\(K^A_r(x)\).

\begin{lemma}
\label{lem:oracle-simulates-rational-ball-code}
For every prefix-free ball code \(\psi\) whose values are rational closed balls,
there is an oracle \(A\subseteq\omega\) such that
\[
K^A(x;r)
\le
K_\psi(x;r)+O(1)
\]
for all \(x\in\mathbb R^n\) and \(r\in\omega\), where the implicit
constant may depend on \(\psi\) but is independent of \(x\) and \(r\).
\end{lemma}

\begin{proof}
Fix computable numberings \((\sigma_m)_{m\in\omega}\) of the finite binary
strings and \((b_k)_{k\in\omega}\) of the rational closed balls, and a computable
pairing function \(\langle\cdot,\cdot\rangle\).  Let \(A\) encode both the domain
and the graph of \(\psi\) by
\[
A=
\{\langle 0,m\rangle:\sigma_m\in\dom(\psi)\}
\cup
\{\langle 1,\langle m,k\rangle\rangle:
\sigma_m\in\dom(\psi)\text{ and }\psi(\sigma_m)=b_k\}.
\]
An oracle machine, on input \(p=\sigma_m\), first checks that \(p\) is marked as
a domain element and that no proper prefix of \(p\) is so marked, and then
searches for the least \(k\) marked as its output.  Its domain is prefix-free
for every oracle, while for this particular \(A\) the marked \(k\) is unique and
the machine computes exactly \(\psi\).
The optimality of \(U\) now gives the stated inequality.
\end{proof}

\begin{theorem}[Point-to-set principles, coding form]
\label{thm:point-to-set-coding}
For every nonempty \(E\subseteq\mathbb R^n\),
\[
\dim_H(E)
=
\min_{A\subseteq\omega}
\sup_{x\in E}
\liminf_{r\to\infty}\frac{K^A(x;r)}{r},
\qquad
\dim_P(E)
=
\min_{A\subseteq\omega}
\sup_{x\in E}
\limsup_{r\to\infty}\frac{K^A(x;r)}{r}.
\]
\end{theorem}

\begin{proof}
Put \(d=\dim_H(E)\).  For every oracle \(A\), the machine \(U^A\) is a
prefix-free ball code, so the classical coding characterization gives
\[
d\le \dim_{U^A}(E)
=
\sup_{x\in E}\liminf_{r\to\infty}\frac{K^A(x;r)}{r}.
\]
Hence
\[
d\le
\inf_{A\subseteq\omega}
\sup_{x\in E}\liminf_{r\to\infty}\frac{K^A(x;r)}{r}.
\]

Conversely, choose a prefix-free ball code \(\varphi\) with
\(\dim_\varphi(E)=d\).  By
Lemma~\ref{lem:rationalize-prefix-free-ball-code}, there is a prefix-free
rational closed ball code \(\psi\) with
\(\dim_\psi(E)\le\dim_\varphi(E)=d\).  On the other hand,
Theorem~\ref{thm:prefix-free-ball-code-hausdorff} gives
\(d\le\dim_\psi(E)\).  Hence \(\dim_\psi(E)=d\).  If \(A\) encodes \(\psi\) as in
Lemma~\ref{lem:oracle-simulates-rational-ball-code}, then
\[
K^A(x;r)\le K_\psi(x;r)+O(1)
\]
for all \(x\) and \(r\).  Dividing by \(r\), taking lower limits, and then
taking the supremum over \(x\in E\) gives
\[
\sup_{x\in E}\liminf_{r\to\infty}\frac{K^A(x;r)}{r}
\le \dim_\psi(E)=d.
\]
Thus equality is attained by this oracle \(A\), and the infimum is a minimum.

The packing equality follows by the same argument, with \(\liminf\), \(\dim\), and
Theorem~\ref{thm:prefix-free-ball-code-hausdorff} replaced by \(\limsup\),
\(\Dim\), and Theorem~\ref{thm:packing-ball-code-dimension}, respectively.
\end{proof}

Thus the point-to-set principles are obtained by effectivizing the classical
ball-code characterizations: in each formula, a minimizing oracle need only
encode a rational closed ball code that is optimal for the set.

\subsection{Point-to-set principles via rational point complexity}
\label{subsec:rational-point-complexity}

We finally compare this coding formulation with the usual Kolmogorov complexity formulation.
In the preliminaries, equations~\eqref{eq:euclidean-point-complexity} and
\eqref{eq:euclidean-effective-dimensions} recalled the nonrelativized
Euclidean point-complexity characterization.
The following definition is the
corresponding oracle-relative version.

\begin{definition}
\label{def:rational-point-dimensions}
Fix an optimal oracle prefix-free Turing machine \(V\) whose outputs are rational
points in \(\mathbb Q^n\).  For an oracle \(A\subseteq\omega\) and
\(q\in\mathbb Q^n\), define
\[
K^A(q)
=
\min\{|p|:V^A(p)=q\}.
\]
For \(x\in\mathbb R^n\) and \(r\in\omega\), define
\[
K^A_r(x)
=
\min\{K^A(q):q\in\mathbb Q^n,\ |x-q|<2^{-r}\}.
\]
Define the effective Hausdorff and packing dimensions of \(x\) relative to
\(A\) by
\[
\dim_e^A(x)
=
\liminf_{r\to\infty}\frac{K^A_r(x)}{r}
\quad\text{and}\quad
\Dim_e^A(x)
=
\limsup_{r\to\infty}\frac{K^A_r(x)}{r}.
\]
For nonempty \(E\subseteq\mathbb R^n\), put
\[
\dim_e^A(E)
=
\sup_{x\in E}\dim_e^A(x)
\quad\text{and}\quad
\Dim_e^A(E)
=
\sup_{x\in E}\Dim_e^A(x).
\]
\end{definition}

Fix also an optimal prefix-free Turing machine whose outputs are natural numbers,
and let \(K_\omega(r)\) be the prefix-free complexity of \(r\in\omega\) with
respect to this machine.  Then \(K_\omega(r)=O(\log(r+2))\).

\begin{proposition}
\label{prop:ball-point-effective-dimensions}
For every oracle \(A\subseteq\omega\), every \(x\in\mathbb R^n\), and every
\(r\in\omega\),
\[
K^A_r(x)\le K^A(x;r)+O(1)
\]
and
\[
K^A(x;r)\le K^A_{r+1}(x)+K_\omega(r)+O(1).
\]
The constants implicit in \(O(1)\) are independent of \(A\), \(x\), and \(r\).
\end{proposition}

\begin{proof}
If \(U^A(p)=\overline B(q,\rho)\) witnesses \(K^A(x;r)\), then
\(|x-q|\le\rho\le 2^{-(r+1)}\), so a program that outputs the center gives the
first inequality.  Conversely, if \(V^A(p)=q\) and
\(|x-q|<2^{-(r+1)}\), then \(p\), together with a shortest description of
\(r\), specifies the ball \(\overline B(q,2^{-(r+1)})\), which contains \(x\)
and has diameter \(2^{-r}\).  The second inequality follows.  In both cases the
additive constant follows from optimality and is uniform in \(A\), \(x\), and
\(r\).
\end{proof}

\begin{corollary}
\label{cor:ball-point-effective-dimensions}
For every oracle \(A\subseteq\omega\) and every \(x\in\mathbb R^n\),
\[
\dim_e^A(x)
=
\liminf_{r\to\infty}\frac{K^A(x;r)}{r},
\qquad
\Dim_e^A(x)
=
\limsup_{r\to\infty}\frac{K^A(x;r)}{r}.
\]
\end{corollary}

\begin{proof}
This follows by dividing the two inequalities in
Proposition~\ref{prop:ball-point-effective-dimensions} by \(r\) and taking lower
and upper limits, since \((r+1)/r\to1\) and \(K_\omega(r)/r\to0\).
\end{proof}

\begin{corollary}[Point-to-set principle]
\label{cor:rational-point-point-to-set}
For every nonempty \(E\subseteq\mathbb R^n\),
\[
\dim_H(E)=\min_{A\subseteq\omega}\sup_{x\in E}\liminf_{r\to\infty}\frac{K^A_r(x)}{r},
\qquad
\dim_P(E)=\min_{A\subseteq\omega}\sup_{x\in E}\limsup_{r\to\infty}\frac{K^A_r(x)}{r}.
\]
\end{corollary}

\begin{proof}
This follows immediately from the coding form of the point-to-set principle and
Corollary~\ref{cor:ball-point-effective-dimensions}.
\end{proof}

Thus Corollary~\ref{cor:rational-point-point-to-set} recovers
Theorem~\ref{thm:lutz-lutz-point-to-set}.

\section*{Statements and Declarations}

\subsection*{Funding}

This work was supported by JSPS KAKENHI Grant Numbers
JP22K03408 
and
JP25K07105. 

\subsection*{Competing Interests}
The author has no relevant financial or non-financial interests to disclose.

\subsection*{Data Availability}
No datasets were generated or analysed during the current study.

\section*{Acknowledgments}

This work was supported by the Research Institute for Mathematical Sciences,
an International Joint Usage/Research Center located in Kyoto University.


\end{document}